\documentclass[12pt]{amsart}
\usepackage{graphics,wrapfig, graphicx, mathrsfs,xcolor}
\usepackage{mathtools}
\usepackage{amssymb,amscd,graphicx,xcolor,subfig,tikz}
\usepackage{graphics}
\usepackage{indentfirst}
\usepackage{bm, enumerate,xcolor}
\usepackage[dvips]{epsfig}
\usepackage{latexsym}
\usepackage{float}
\usepackage[colorlinks]{hyperref}
\AtBeginDocument{
   \hypersetup{
    linkcolor=blue,
    citecolor=blue,
 }
}

\usepackage{amsmath}
\usepackage{amssymb}
\usepackage[applemac]{inputenc}
\usepackage[T1]{fontenc}
\newtheorem{lemma}{Lemma}[section]
\newtheorem{theorem}{Theorem}[section]

\newtheorem{definition}{Definition}[section]
\newtheorem{remark}{Remark}[section]

\numberwithin{equation}{section} \numberwithin{theorem}{section}
\numberwithin{example}{section} \numberwithin{remark}{section}
\numberwithin{figure}{section} \numberwithin{algorithm}{section}
\mathtoolsset{showonlyrefs}
\def\ds{\displaystyle}
\begin{document}
\title[Density estimates]{A further remark on the density estimate for degenerate Allen-Cahn equations: $\Delta_{p}$-type equations for $1<p<\frac{n}{n-1}$ with rough coefficients}

\author{Chilin Zhang}
\address{School of Mathematical Sciences, Fudan University, Shanghai 200433, China}\email{zhangchilin@fudan.edu.cn}

\begin{abstract}
    In this short remark on a previous paper \cite{SZ25}, we continue the study of Allen-Cahn equations associated with Ginzburg-Landau energies
    \begin{equation*}
        J(v,\Omega)=\int_{\Omega}\Big\{F(\nabla v,v,x)+W(v,x)\Big\}dx,
    \end{equation*}
    involving a Dirichlet energy $F(\vec{\xi},\tau,x)\sim|\vec{\xi}|^{p}$ and a degenerate double-well potential $W(\tau,x)\sim(1-\tau^{2})^{m}$. In contrast to \cite{SZ25}, we remove all regularity assumptions on the Ginzburg-Landau energy. Then, with further assumptions that $1<p<\frac{n}{n-1}$ and that $W(\tau,x)$ is monotone in $\tau$ on both sides of $0$, we establish a density estimate for the level sets of nontrivial minimizers $|u|\leq1$.
\end{abstract}

\maketitle

\section{Introduction}
The Ginzburg-Landau energy was developed from the theory of Van der Waals (see \cite{Row79}) by Landau, Ginzburg and Pitaevski\u i in \cite{GP58,Landau37,Landau67} to describe phase transitions in thermodynamics (see also \cite{AC72,CH58}). In this paper, we study a global minimizer $u:\mathbb{R}^{n}\to[-1,1]$ of the Ginzburg-Landau energy
\begin{equation}\label{eq. GL energy}
    J(v,\Omega)=\int_{\Omega}\Big\{F(\nabla v,v,x)+W(v,x)\Big\}dx,
\end{equation}
in which $v$ represents the mean field of the spin of the particles. A minimizer of \eqref{eq. GL energy} is defined as follows:
\begin{definition}
    Let $\Omega\subseteq\mathbb{R}^{n}$. We say that $u\in W^{1,p}(\Omega,[-1,1])$ is a minimizer of \eqref{eq. GL energy}, if for every bounded open set $\Omega' \subset\subset \Omega$, and any $v\in W^{1,p}_{loc}(\Omega,[-1,1])$ such that $u=v$ in $\Omega\setminus \Omega'$, we have
    \begin{equation*}
        J(u,\Omega')\leq J(v,\Omega').
    \end{equation*}
\end{definition}
\begin{remark}
    When $F(\vec{\xi},\tau,x)$ and $W(\tau,x)$ have good regularity, in particular, when
\begin{equation*}
    F(\vec{\xi},\tau,x)\equiv|\vec{\xi}|^{p},\quad W(\tau,x)\equiv(1-v^{2})^{m},
\end{equation*}
then the Euler-Lagrange equation of a minimizer, namely the Allen-Cahn equation, is the following:
\begin{equation}\label{eq. allen cahn equation}
    p\Delta_{p}u=p\cdot\mathrm{div}(|\nabla u|^{p-2}\nabla u)=W'(u)=-2m(1-u^{2})^{m-1}u.
\end{equation}
\end{remark}
However, in this paper, we avoid the use of the Euler-Lagrange equation because we do not impose any regularity assumptions. Instead, we assume that there exists some universal constant $\lambda>1$ such that:
\begin{itemize}
    \item[(A)] For every vector $\vec{\xi}$, $\tau\in[-1,1]$ and $x\in\mathbb{R}^{n}$, we have
    \begin{equation*}
        \lambda^{-1}|\vec{\xi}|^{p}\leq F(\vec{\xi},\tau,x)\leq\lambda|\vec{\xi}|^{p},
    \end{equation*}
    where the exponent $p$ is universal and it satisfies $1<p<\frac{n}{n-1}$;
    \item[(B)] For every $\tau\in[-1,1]$, we have that
    \begin{equation*}
        \lambda^{-1}(1-\tau^{2})^{m}\leq W(\tau,x)\leq\lambda(1-\tau^{2})^{m},
    \end{equation*}
    where the exponent $m$ is universal and it satisfies $m>p$;
    \item[(C)] For each fixed $x$, $W(\tau,x)$ is increasing when $-1\leq\tau\leq0$, and it is decreasing when $0\leq\tau\leq1$.
\end{itemize}

Apart from the trivial minimizers $u\equiv\pm1$, the more complicated and interesting question is to study minimizers or critical points representing phase transitions, i.e. a solution that can be sufficiently close to both $1$ and $-1$ (two steady states), but with a phase field region $|u|\leq1-\varepsilon$ in between.

It is well known that phase transitions modeled by minimizers $u$ defined in a large ball $B_R$ are closely related to minimal surfaces. More precisely, the rescaling of the transition region $\frac 1R \{|u| \le 1-\varepsilon\}$ of $u$ from $B_R$ to the unit ball $B_1$ is well approximated by a minimal surface in $B_1$. In the classical case $p=m=2$ the approximation is made rigorous in three main steps:

\begin{itemize}
    \item[(1)] The $\Gamma$-convergence result established by Modica and Mortola in \cite{Modica,MM77}, see also \cite{Bou90,OS91}.
    \item[(2)] The density estimate obtained by Caffarelli-C\'ordoba in \cite{CC95}, see also \cite{DFV18,FV08,PV05b,PV05a,V04}.
    \item[(3)] The convergence of $\{u_R=0\}$ to $\partial E$ in the stronger $C^{2,\alpha}_{loc}(B_1)$ sense, i.e.: the improvement of flatness technique, see \cite{S09,SV05,VSS06}.
\end{itemize}

The heteroclinical solution is the monotone one-dimensional solution that connects the stable phases $- 1$ and $+1$ as $x$ ranges from $-\infty$ to $ \infty$. The rate of decay of this solution to the limits $\pm 1$ depends on the values of $m$ and $p$. Precisely, if $m <p$ then the limits are achieved outside a finite interval, producing a free boundary of Alt-Phillips type for the region $u\neq\pm1$. If $m=p$ the rate of decay is exponential. On the other hand, the case $m> p$ produces less stable minimal points (still at $\pm1$) for an infinitesimal potential energy $W(v)\sim(1-v^{2})^{m}$, and the rate of decay is polynomial. To see this, one can multiply $u'(t)$ on both sides of the one-dimensional version of \eqref{eq. allen cahn equation} and integrate. It follows that the heteroclinical solution satisfies the first-order ODE
\begin{equation*}
    u'(t)=\Big(\frac{(1-u^{2})^{m}}{p-1}\Big)^{1/p}.
\end{equation*}
The decay rate is then obtained by integrating the equation $\ds\frac{u'(t)}{(1-u^{2})^{m/p}}\sim1$. Recently, the decay rate estimate for the heteroclinical solution has been extended to the nonlocal Allen-Cahn equation in De Pas-Dipierro-Piccinini-Valdinoci \cite{dPDPV25}. It is then natural to investigate whether the results of $\Gamma$-convergence, density estimate, and improvement of flatness mentioned above extend to these types of degenerate Ginzburg-Landau energies.

In \cite{DFV18}, Dipierro-Farina-Valdinoci considered $Q$-minimizers (a relaxation of the terminology minimizer) of such degenerate energies and obtained the density estimates for a certain range of $m$'s depending on the dimension $n$. Precisely, the authors considered general Ginzburg-Landau type energies
\begin{equation}\label{eq. DFV18 assumption}
    J(v)=\int_{\Omega}E(\nabla v,v,x)dx,\quad E(\vec{\xi},\tau,x)\sim(|\vec{\xi}|^{p}+|\tau+1|^{m}).
\end{equation}
If $\frac{pm}{m-p}>n$ and $\Big|\{u\geq0\}\cap B_{1}\Big|>c$ for some positive $c$, then the authors showed that
\begin{equation*}
    \Big|\{u\geq0\}\cap B_{R}\Big|\geq\delta R^{n},\quad\mbox{for some }\delta>0\mbox{ depending on }E(\cdot,\cdot,\cdot)\mbox{ and }c.
\end{equation*}
Notice that the energy \eqref{eq. GL energy} with assumptions $\mathrm{(A)(B)(C)}$ on $F(\vec{\xi},\tau,x)$ and $W(\tau,x)$ is a special case of the energy \eqref{eq. DFV18 assumption}. We also remark that the density estimates in the non-degenerate case $0<m\leq p$ were obtained in the earlier work Farina-Valdinoci \cite{FV08}.

In Savin-Zhang \cite{SZ25}, by further assuming that there exists an initial ball $B_\rho$ of a fixed large radius in which the density estimate holds (see \eqref{eq. additional assumption} below), the authors removed the assumption $\frac{pm}{m-p}>n$ and obtained a new version of the density estimate in \cite[Theorem 1.1]{SZ25}. We state its simplified version below as a lemma:
\begin{lemma}\label{lem. SZ25}
    Let $u$ be a minimizer of the energy \eqref{eq. DFV18 assumption} in $\mathbb{R}^{n}$. Given any $\varepsilon >0$, there exist $r_{0}=r_{0}(\varepsilon)$ large and $\delta=\delta(\varepsilon)$, so that if
    \begin{equation}\label{eq. additional assumption}
        \Big|\{u\geq0\}\cap B_{r}\Big|\geq\varepsilon r^{n}
    \end{equation}
    for some $r\geq r_{0}$, then $\Big|\{u\geq0\}\cap B_{R}\Big|\geq\delta R^{n}$, for all $R\geq r$.
\end{lemma}
As an application, the authors proved the density estimate for a class of degenerate Allen-Cahn equations with further regularity assumptions on $F$ and $W$. Their strategy is to translate the origin to a specific point $x^{*}$ where $u(x^{*})$ is sufficiently close to $1$, and to verify the condition \eqref{eq. additional assumption} there. Recently, in \cite{DFGV25a,DFGV25b}, Dipierro-Farina-Giacomin-Valdinoci have obtained a similar result in the nonlocal setting.

In this paper, the main result is the following density estimate for a class of degenerate Allen-Cahn equations, whose Ginzburg-Landau energy has little regularity.
\begin{theorem}\label{thm. main}
    Assume that $u:\mathbb{R}^{n} \to[-1,1]$ is a minimizer of the energy \eqref{eq. GL energy} in $\mathbb{R}^{n}$, such that $F(\vec{\xi},\tau,x)$ and $W(\tau,x)$ satisfy assumptions $\mathrm{(A)(B)(C)}$. If $u(0)=0$, then there exist some universal constants $\delta,R_{0}>0$, such that for every $R\geq R_{0}$, we have
    \begin{equation*}
        \Big|B_{R}\cap\{u\geq0\}\Big|\geq\delta R^{n}\mbox{ and }\Big|B_{R}\cap\{u\leq0\}\Big|\geq\delta R^{n}.
    \end{equation*}
\end{theorem}
Similar to \cite{SZ25}, the strategy in proving Theorem~\ref{thm. main} is to translate the origin to some point $x^{*}$ and to verify the assumption \eqref{eq. additional assumption} in Lemma~\ref{lem. SZ25}. In the first step, we prove that $\ds\max_{B_{R}}u=u(x^{*})$ is sufficiently close to $1$ for a uniform radius $R$. In the second step, we prove that the density of the positive set is large in a ball centered at $x^{*}$, thus verifying the condition \eqref{eq. additional assumption}. The two key steps mentioned above are derived via variants of the weak Harnack principles.

\section{Proof of Theorem~\ref{thm. main}}
As was mentioned in the Introduction, it suffices to verify that the assumption \eqref{eq. additional assumption} is satisfied in a fixed ball close to the origin. As a preliminary lemma, we prove the energy estimate of a minimizer: 
\begin{lemma}\label{lem. energy estimate}
    Let $u$ be a minimizer to \eqref{eq. GL energy} satisfying the assumptions $\mathrm{(A)(B)(C)}$. Then there exists some universal constant $C$, such that $J(u,B_{R})\leq C R^{n-1}$ for all $R\geq1$.
\end{lemma}
\begin{remark}
    In fact, such an energy estimate was already proven in \cite{DFV18}, and it holds true both not only when the Ginzburg-Landau energy \eqref{eq. GL energy} is degenerate ($m>p$), but also when it is non-degenerate ($m\leq p$).
\end{remark}
\begin{proof}[Proof of Lemma~\ref{lem. energy estimate}]
    Let $v(x)=v(|x|)$, such that ("$\mathrm{med}$" stands for the median of the three quantities):
    \begin{equation*}
        v(r)=\mathrm{med}(-1,1,r-R-1).
    \end{equation*}
    Let $\Omega=\{u\geq v\}$, then we have $B_{R}\subseteq\Omega$ and $\overline{\Omega}\subseteq B_{R+2}$. By the minimality of $u$, we have
    \begin{equation*}
        J(u,B_{R})\leq J(u,\Omega)\leq J(v,\Omega)\leq J(v,B_{R+2})\leq C R^{n-1}.
    \end{equation*}
    Here, in the last step of the inequality above, we have used the fact that the infinitesimal energy
    \begin{equation*}
        F(\nabla v,v,x)+W(v,x)\leq C
    \end{equation*}
    for some uniform constant $C$, and that it is supported only in the annulus $B_{R+2}\setminus B_{R}$.
\end{proof}
Now, let us prove Lemma~\ref{lem. key lemma 1} and Lemma~\ref{lem. key lemma 2}, which are the key steps to proving Theorem~\ref{thm. main}. During the proof, $C$ denotes some universal constant, which might change from line to line.

In the first key step, we show that $u$ is close to $1$ at some $x^{*}\in B_{\rho}$ for some sufficiently large $\rho$.
\begin{lemma}\label{lem. key lemma 1}
    Let $u$ with $u(0)=0$ be a minimizer to \eqref{eq. GL energy} satisfying the assumptions $\mathrm{(A)(B)(C)}$. Given any $h<1$, then there exists some $\rho=\rho(h)$ such that $\ds\max_{B_{\rho}}u\geq1-h$.
\end{lemma}
\begin{proof}
    Let us assume that $\ds\max_{B_{R}}u\leq1-h$ for some radius $R=2^{L}$ (without loss of generality, assume that $L\geq2$ is an integer), then it suffices to find an upper bound of $L$ depending on $h$. For the given $h$, we can choose a fixed $t_{\infty}=t_{\infty}(h)$, such that:
    \begin{equation*}
        -1<t_{\infty}<0,\quad W(t_{\infty})\leq W(1-h).
    \end{equation*}
    For all $k\geq0$, we make the following notations:
    \begin{equation*}
        t_{k}=(1-2^{-k-1})t_{\infty}-2^{-k-1},\quad r_{k}=\frac{1+2^{-k}}{2}R,\quad A_{k}=B_{r_{k}}\cap\{u\geq t_{k}\}.
    \end{equation*}
    
    We consider a sequence of competitors $\phi_{k}$. When $k\geq L-1$, we let
    \begin{equation*}
        \phi_{k}(x)=\mathrm{med}\Big(t_{k},1,1+\frac{2^{k+2}}{R}(1-t_{k})(|x|-r_{k})\Big).
    \end{equation*}
    It follows that $\phi_{k}\equiv1$ outside $B_{r_{k}}$ and $\phi_{k}\equiv t_{k}$ inside $B_{r_{k+1}}$. Besides, as $|\nabla\phi_{k}|\leq C\frac{2^{k}}{R}$ in the annulus $B_{r_{k}}\setminus B_{r_{k+1}}$, we use the assumptions (A)(B) on $F(\cdot,\cdot,\cdot)$ and $W(\cdot,\cdot)$, and have that
    \begin{equation}\label{eq. F+W estimate when k>=L-1}
        F(\nabla\phi_{k},\phi_{k},x)+W(\phi_{k},x)\leq\lambda\Big(C\frac{2^{k}}{R}\Big)^{p}+\lambda\leq C\frac{2^{kp}}{R^{p}}\quad\mbox{in the annulus }B_{r_{k}}\setminus B_{r_{k+1}},
    \end{equation}
    where we have used $\frac{2^{k}}{R}=\frac{2^{k}}{2^{L}}\geq\frac{1}{2}$ in the last step of the inequality above.
    
    When $0\leq k\leq L-2$, we choose some $N_{k}\in(r_{k+1},r_{k}]\cap\mathbb{Z}$ (the choice of $N_{k}$ will be specified later), or equivalently:
    \begin{equation*}
        N_{k}\in\{2^{L-1}+2^{L-k-2}+1,2^{L-1}+2^{L-k-2}+2,\cdots,2^{L-1}+2^{L-k-1-2}-1,2^{L-1}+2^{L-k-1-2}\}.
    \end{equation*}
    With such a choice of $N_{k}$, we set
    \begin{equation*}
        \phi_{k}(x)=\mathrm{med}\Big(t_{k},1,1+(1-t_{k})(|x|-N_{k})\Big).
    \end{equation*}
    We then have $\phi_{k}\equiv1$ outside $B_{N_{k}}$, $\phi_{k}\equiv t_{k}$ inside $B_{N_{k}-1}$. Besides, as $|\nabla\phi_{k}|\leq2$ in the annulus $B_{r_{k}}\setminus B_{r_{k+1}}$, we use the assumptions (A)(B) on $F(\cdot,\cdot,\cdot)$ and $W(\cdot,\cdot)$, and have that
    \begin{equation*}
        F(\nabla\phi_{k},\phi_{k},x)+W(\phi_{k},x)\leq\lambda\cdot2^{p}+\lambda\leq C\quad\mbox{in the annulus }B_{N_{k}}\setminus B_{N_{k}-1}.
    \end{equation*}

    Denote $\Omega_{k}=\{u>\phi_{k}\}$ for $k\geq0$, then since $u\leq1-h$ in $B_{R}$, we see $\overline{\Omega_{k}}\subseteq B_{r_{k}}$ for $k\geq L-1$ and $\overline{\Omega_{k}}\subseteq B_{N_{k}}\subseteq B_{r_{k}}$ for $0\leq k\leq L-2$. It then follows from the minimality of $u$ that:
    \begin{align*}
        \lambda^{-1}\int_{\Omega_{k}}|\nabla u|^{p}dx\leq&J(u,\Omega_{k})-\int_{\Omega_{k}}W(u,x)dx\leq J(\phi_{k},\Omega)-\int_{\Omega_{k}}W(u,x)dx\\
        \leq&\lambda\int_{\Omega_{k}}|\nabla\phi_{k}|^{p}dx+\int_{\Omega_{k}}\Big\{W(\phi_{k},x)-W(u,x)\Big\}dx.
    \end{align*}
    Since $2^{1-p}|\vec{\xi}-\vec{\eta}|^{p}\leq|\vec{\xi}|^{p}+|\vec{\eta}|^{p}$ for any two vectors, we choose $\vec{\xi}=\nabla u$, $\vec{\eta}=\nabla\phi_{k}$, and conclude that
    \begin{equation}\label{eq. u-phi_k W^1,p estimate}
        \int_{\Omega_{k}}|\nabla(u-\phi_{k})|^{p}dx\leq C\int_{\Omega_{k}}|\nabla\phi_{k}|^{p}dx+C\int_{\Omega_{k}}\Big\{W(\phi_{k},x)-W(u,x)\Big\}dx.
    \end{equation}
    
    Since $W(t_{\infty})\leq W(1-h)$ and since $u\leq1-h$ in $B_{R}$, by the assumption (C) of $W(\tau,x)$, we see that $W(\phi_{k},x)\leq W(u,x)$ when $x\in\Omega_{k}$ and $\phi_{k}=t_{k}\leq t_{\infty}$. Besides, we have $\Omega_{k}\subseteq A_{k}$. Then, we have:
    \begin{itemize}
        \item Case 1: If $0\leq k\leq L-2$, then:
        \begin{align*}
            \int_{\Omega_{k}}|\nabla(u-\phi_{k})|^{p}dx\leq&C\int_{(B_{N_{k}}\setminus B_{N_{k}-1})\cap A_{k}}|\nabla\phi_{k}|^{p}dx+C\int_{(B_{N_{k}}\setminus B_{N_{k}-1})\cap A_{k}}W(\phi_{k},x)dx\\
            \leq&C|(B_{N_{k}}\setminus B_{N_{k}-1})\cap A_{k}|.
        \end{align*}
        Moreover, if we choose $N_{k}$ wisely, then we even have the following estimate:
        \begin{equation}\label{eq. small k good choice}
            \int_{\Omega_{k}}|\nabla(u-\phi_{k})|^{p}dx\leq C\frac{2^{k}}{R}|A_{k}|.
        \end{equation}
        In fact, since the width of the annulus $B_{r_{k}}\setminus B_{r_{k+1}}$ equals $2^{L-k-2}=\frac{R}{2^{k+2}}$, and that
        \begin{equation*}
            \sum_{N=r_{k+1}+1}^{r_{k}}|(B_{N}\setminus B_{N-1})\cap A_{k}|=|(B_{r_{k}}\setminus B_{r_{k+1}})\cap A_{k}|\leq|A_{k}|,
        \end{equation*}
        we can then choose some $N_{k}\in(r_{k+1},r_{k}]\cap\mathbb{Z}$, such that $|(B_{N_{k}}\setminus B_{N_{k}-1})\cap A_{k}|\leq\frac{2^{k+2}}{R}|A_{k}|$.
        \item Case 2: If $k\geq L-1$, using the estimate \eqref{eq. F+W estimate when k>=L-1}, we obtain from \eqref{eq. u-phi_k W^1,p estimate} that:
        \begin{equation}\label{eq. large k good estimate}
            \int_{\Omega_{k}}|\nabla(u-\phi_{k})|^{p}dx\leq C\int_{\Omega_{k}}\Big\{|\nabla\phi_{k}|^{p}+W(\phi_{k},x)\Big\}dx\leq C\frac{2^{kp}}{R^{p}}|A_{k}|.
        \end{equation}
    \end{itemize}
    
    Now, we first apply the Sobolev inequality, and then apply the H\"older inequality, and obtain that:
    \begin{equation*}
        \int_{\Omega_{k}}|\nabla(u-\phi_{k})|^{p}dx\geq C\Big\{\int_{\Omega_{k}}|u-\phi_{k}|^{\frac{np}{n-p}}dx\Big\}^\frac{n-p}{n}\geq C|\Omega_{k}|^{-\frac{p}{n}}\int_{\Omega_{k}}|u-\phi_{k}|^{p}dx.
    \end{equation*}
    Recall that $\Omega_{k}\subseteq A_{k}$ and that $\phi_{k}\equiv t_{k}$ in $B_{r_{k+1}}$, then we conclude that
    \begin{equation}\label{eq. Sobolev LHS}
        \int_{\Omega_{k}}|\nabla(u-\phi_{k})|^{p}dx\geq C|A_{k}|^{-\frac{p}{n}}\int_{\Omega_{k}\cap B_{r_{k+1}}\cap\{u\geq t_{k+1}\}}|u-t_{k}|^{p}dx\geq C|A_{k}|^{-\frac{p}{n}}\frac{|A_{k+1}|}{2^{kp}}.
    \end{equation}
    We combine \eqref{eq. small k good choice}, \eqref{eq. large k good estimate}, and \eqref{eq. Sobolev LHS} and get
    \begin{align*}
        |A_{k+1}|\leq&C\frac{2^{k(1+p)}}{R}|A_{k}|^{1+\frac{p}{n}},\quad\mbox{if }0\leq k\leq L-2,\\
        |A_{k+1}|\leq&C\frac{4^{kp}}{R^{p}}|A_{k}|^{1+\frac{p}{n}},\quad\mbox{if }k\geq L-1.
    \end{align*}
    Consequently, by using the observation that $\frac{4^{kp}}{R^{p}}\geq2^{1-p}\cdot\frac{2^{k(1+p)}}{R}$ for $k\geq L-1$, we have the following inductive inequality for $|A_{k}|$'s:
    \begin{equation}\label{eq. weak harnack iteration 1}
        |A_{k+1}|\leq C\frac{2^{k(1+p)}}{R}|A_{k}|^{1+\frac{p}{n}},\quad\mbox{for all }k\geq0.
    \end{equation}

    Now we divide $R^{\frac{n}{p}}$ on both sides of \eqref{eq. weak harnack iteration 1}, and denote $\beta_{k}=R^{-\frac{n}{p}}|A_{k}|$, then we have
    \begin{equation}\label{eq. weak harnack iteration 2}
        \beta_{k+1}\leq C2^{k(1+p)}\beta_{k}^{1+\frac{p}{n}}.
    \end{equation}
    Notice that for all $x\in A_{0}\subseteq B_{R}$, $u(x)\geq t_{0}=\frac{t_{\infty}-1}{2}$ and $u(x)\leq1-h$, then by Lemma~\ref{lem. energy estimate}, we have
    \begin{equation*}
        \beta_{0}=R^{-\frac{n}{p}}|A_{0}|\leq R^{-\frac{n}{p}}\frac{1}{\min\{\lambda|\frac{t_{\infty}+1}{2}|^{m},\lambda h^{m}\}}\int_{B_{R}}W(u,x)dx\leq R^{-\frac{n}{p}}\cdot\frac{C R^{n-1}}{c(h)}.
    \end{equation*}
    When $1<p<\frac{n}{n-1}$, we have $n-1-\frac{n}{p}<0$. By choosing $\rho=\rho(h)$ sufficiently large, we see that $\ds\max_{B_{\rho}}u\geq1-h$. In fact, suppose on the contrary that $\ds\max_{B_{R}}u\leq1-h$ for some large $R$. Then, as
    \begin{equation*}
        \lim_{R\to+\infty}\frac{C R^{n-1-\frac{n}{p}}}{c(h)}=0,
    \end{equation*}
    we see that $\beta_{0}$ is sufficiently small. In other words, $\ln{\beta_{0}}$ is a negative number with a sufficiently large absolute value. We take the logarithm on both sides of the inductive inequality \eqref{eq. weak harnack iteration 2}, and have that
    \begin{equation*}
        \ln{\beta_{k+1}}\leq(1+\frac{p}{n})\ln{\beta_{k}}+k\cdot\ln{(2^{1+p})}+\ln{C}.
    \end{equation*}
    As the initial data $\ln{\beta_{0}}$ is a sufficiently large negative number, we can inductively show that
    \begin{equation*}
        \ln{\beta_{k}}\leq(1+\frac{p}{2n})^{k}\cdot\ln{\beta_{0}}.
    \end{equation*}
    As a result, $\ln{\beta_{k}}\to-\infty$ as $k\to\infty$, or equivalently, $\beta_{k}\to0$. In other words, we have that $u\leq t_{\infty}$ almost everywhere in $B_{R/2}$, which contradicts the assumption $u(0)=0$.
\end{proof}

In the second key step, we show that the positive set of $u$ is large near some point $x^{*}\in B_{\rho}$.
\begin{lemma}\label{lem. key lemma 2}
    There exists some universal constant $\sigma>0$ and a function $h=h(R)>0$ for all $R\geq1$, such that the following holds: Assume that $u$ is a minimizer to \eqref{eq. GL energy} satisfying the assumptions $\mathrm{(A)(B)(C)}$ and that $u(0)\geq1-h$, then $\Big|B_{R}\cap\{u\geq0\}\Big|\geq\sigma R^{n}$.
\end{lemma}
\begin{proof}
    As $W(\tau,x)\leq\lambda(1-\tau^{2})^{m}$ with $m>p$, we choose $h=h(R)=\min\Big\{(2^{m}\lambda R^{p})^{-\frac{1}{m-p}},\frac{1}{2}\Big\}$, then
    \begin{equation*}
        W(\tau,x)\leq h^{p}R^{-p},\quad\mbox{for all }1-2h\leq\tau\leq1.
    \end{equation*}
    For each $h\leq a\leq 2h$, we consider a competitor
    \begin{equation*}
        \phi_{a}(x)=\min\Big\{(1-a)+\frac{4h^{2}|x|^{2}}{a R^{2}},1\Big\}.
    \end{equation*}
    We can easily verify that $\{\phi_{a}(x)<1\}\subseteq B_{R}$ and $|\nabla\phi_{a}|\leq\frac{8h}{R}$ everywhere. Moreover, $W(\phi_{a},x)\leq h^{p}R^{-p}$ everywhere since $\phi_{a}\geq1-2h$. Now let us denote
    \begin{equation*}
        \Omega_{a}=\{u>\phi_{a}\},\quad V_{a}=\int_{\Omega_{a}}(u-\phi_{a})dx.
    \end{equation*}
    We clearly have $\overline{\Omega_{a}}\subseteq B_{R}$. Then, we deduce from the minimality of $u$ that
    \begin{equation*}
        \lambda^{-1}\int_{\Omega_{a}}|\nabla u|^{p}dx\leq J(u,\Omega_{a})\leq J(\phi_{a},\Omega_{a})\leq\lambda\int_{\Omega_{a}}|\nabla\phi_{a}|^{p}dx+\int_{\Omega_{a}}W(\phi_{a},x)dx.
    \end{equation*}
    Since $2^{1-p}|\vec{\xi}-\vec{\eta}|^{p}\leq|\vec{\xi}|^{p}+|\vec{\eta}|^{p}$ for any two vectors, we choose $\vec{\xi}=\nabla u$, $\vec{\eta}=\nabla\phi_{a}$. By applying the H\"older inequality to the function $|\nabla(u-\phi_{a})|$, we conclude that:
    \begin{equation*}
        \Big([u-\phi_{a}]_{W^{1,1}(\Omega_{a})}\Big)^{p}|\Omega_{a}|^{1-p}\leq\int_{\Omega_{a}}|\nabla(u-\phi_{a})|^{p}dx\leq C\int_{\Omega_{a}}\Big\{|\nabla\phi_{a}|^{p}dx+W(\phi_{a},x)\Big\}dx\leq C h^{p}R^{-p}|\Omega_{a}|.
    \end{equation*}
    By the H\"older inequality and the Sobolev inequality, we have
    \begin{equation}\label{eq. Sobolev for VR}
        V_{a}\leq\|u-\phi_{a}\|_{L^{\frac{n}{n-1}}(\Omega_{a})}|\Omega_{a}|^{\frac{1}{n}}\leq[u-\phi_{a}]_{W^{1,1}(\Omega_{a})}|\Omega_{a}|^{\frac{1}{n}}\leq\frac{Ch}{R}|\Omega_{a}|^{1+\frac{1}{n}}.
    \end{equation}
    
    On the other hand, note that $\frac{d}{d\kappa}\Big|_{\kappa=a}\phi_{\kappa}(x)\leq-1$ for any $x\in\Omega_{a}$, we then have
    \begin{equation}\label{eq. ODE 1}
        \frac{d}{d\kappa}\Big|_{\kappa=a}V_{\kappa}=-\int_{\Omega_{a}}\frac{d}{d\kappa}\Big|_{\kappa=a}\phi_{\kappa}(x)dx\geq|\Omega_{a}|\geq c(\frac{R}{h}\cdot V_{a})^{\frac{n}{n+1}},
    \end{equation}
    where we have used \eqref{eq. Sobolev for VR} in the last step. Recall that the assumption $u(0)\geq1-h$ implies that $V_{a}>0$ for all $a>h$, we can divide $V_{a}^{\frac{n}{n+1}}$ on both sides of \eqref{eq. ODE 1} when $a>h$, and obtain the following:
    \begin{equation*}
        \frac{d}{d\kappa}\Big|_{\kappa=a}V_{\kappa}^{\frac{1}{n+1}}\geq c(\frac{R}{h})^{\frac{n}{n+1}},\quad\mbox{for all }h<a\leq2h,
    \end{equation*}
    which implies that $V_{2h}\geq c R^{n}h$. Notice that $u-\phi_{2h}\leq2h$ and $u\geq0$ in $\Omega_{2h}$, we then have
    \begin{equation*}
        \Big|B_{R}\cap\{u\geq0\}\Big|\geq|\Omega_{2h}|\geq\frac{V_{2h}}{2h}\geq c R^{n}.
    \end{equation*}
    This proves the existence of the desired uniform constant $\sigma>0$.
\end{proof}

With Lemma~\ref{lem. key lemma 1} and Lemma~\ref{lem. key lemma 2}, we now prove the main result.
\begin{proof}[Proof of Theorem~\ref{thm. main}]
    Let us choose the density $\varepsilon$ in Lemma~\ref{lem. SZ25} as the universal constant $\sigma$ in Lemma~\ref{lem. key lemma 2}. Using the two functions $r_{0}(\cdot)$ and $\delta(\cdot)$ obtained in Lemma~\ref{lem. SZ25}, we set $\widetilde{r}=\max\{r_{0}(\sigma),1\}$ and $\widetilde{\delta}=\delta(\sigma)$. Using the function $h(\cdot)$ obtained in Lemma~\ref{lem. key lemma 2}, we set $\widetilde{h}=h(\widetilde{r})$. Using the function $\rho(\cdot)$ obtained in Lemma~\ref{lem. key lemma 1}, we set $\widetilde{\rho}=\rho(\widetilde{h})$. Note that as the constant $\sigma$ from Lemma~\ref{lem. key lemma 2} is universal, all other constants $\widetilde{r},\widetilde{\delta},\widetilde{h},\widetilde{\rho}$ above are also universal.

    By Lemma~\ref{lem. key lemma 1}, there exists some $x^{*}\in B_{\widetilde{\rho}}$ such that $u(x^{*})=1-\widetilde{h}$. Then, we apply Lemma~\ref{lem. key lemma 2} to the translated function $u(x-x^{*})$, and conclude that
    \begin{equation*}
        \Big|B_{\widetilde{r}}(x^{*})\cap\{u\geq0\}\Big|\geq\sigma\widetilde{r}^{n}.
    \end{equation*}
    As $\widetilde{r}\geq r_{0}(\sigma)$, we apply Lemma~\ref{lem. SZ25} to $u(x-x^{*})$, and obtain the following:
    \begin{equation*}
        \Big|B_{r}(x^{*})\cap\{u\geq0\}\Big|\geq\widetilde{\delta}r^{n}\quad\mbox{for all }r\geq\widetilde{r}.
    \end{equation*}
    
    Finally, we choose the universal constants $\delta$ and $R_{0}$ in Theorem~\ref{thm. main}. In fact, we set
    \begin{equation*}
        R_{0}=2(\widetilde{\rho}+\widetilde{r}),\quad\delta=\frac{\sigma}{2^{n}}.
    \end{equation*}
    It then follows that for all $R\geq R_{0}$, we have $B_{R/2}(x^{*})\subseteq B_{R}$ and $R/2\geq\widetilde{r}$. Then,
    \begin{equation*}
        \Big|B_{R}\cap\{u\geq0\}\Big|\geq\Big|B_{R/2}(x^{*})\cap\{u\geq0\}\Big|\geq\sigma(\frac{R}{2})^{n}=\delta R^{n}.
    \end{equation*}
    The other inequality $\Big|B_{R}\cap\{u\leq0\}\Big|\geq\delta R^{n}$ can be argued similarly. Therefore, we have finished the proof of Theorem~\ref{thm. main}.
\end{proof}

\end{document}